\documentclass{amsart}

\newtheorem{thm}{Theorem}[section]

\newtheorem{lem}{Lemma}[section]

\newtheorem{col}{Corollary}[section]
\newtheorem{prob}{Problem}[section]

\numberwithin{equation}{section}

\begin{document}

\title{Fej\'{e}r-type positive operator based on Takenaka--Malmquist system on unit circle}

\author{F.G. Abdullayev}
\address{Kyrgyz-Turkish Manas University, Bishkek, Kyrgyz Republic; Mersin University, Turkey}
\email{fahreddin.abdullayev@manas.edu.kg}
\thanks{Research supported by the Kyrgyz-Turkish Manas University (Bishkek/Kyrgyz Republic), Project No. KTMU-BAP-20I9.FBE,06.}

\author{ V.V. Savchuk}
\address{Institute of Mathematics of NAS of Ukraine, Kyiv, Ukraine}
\email{savchuk@imath.kiev.ua}





\begin{abstract}
	Let $\varphi=\{\varphi_k\}_{k=-\infty}^\infty$ denote the extended Takenaka--Malmquist system on unit circle $\mathbb T$ and let $\sigma_{n,\varphi}(f),$ $f\in L^1(\mathbb T)$, be the Fej\'er-type operator based on $\varphi$, introduced by V. N. Rusak.  We  give the convergence criteria for $\sigma_{n,\varphi}(f)$ in Banach space $X(\mathbb T):=L^p(\mathbb T)\vee C(\mathbb T)$, $p\ge 1$.   Also we prove the Voronovskaya-type theorem for $\sigma_{n,\varphi}(f)$ on class of holomorphic functions representable by Cauchy-type integrals with bounded densities.
\end{abstract}
	
\maketitle	
		
	
2000 MS Classification: {\it 30C10, 30E10, 41A20, 41A36, 42C05}

\bigskip 

\section{Introduction}

Let $C(\mathbb T)$ denote a Banach space of continuous functions on the unit circle $\mathbb T:=\{z\in\mathbb C : |z|=1\}$ equipped with the norm $\|f\|_{C(\mathbb T)}=\max_{t\in\mathbb T}|f(t)|.$

Let a function $f\in C(\mathbb T)$ and let
\[
f(\mathrm e^{\mathrm ix})\sim\sum_{k\in\mathbb Z}c_k\mathrm e^{\mathrm ikx}
\]
be its trigonometric Fourier series. The Ces\`{a}ro $(C,1)$ means of $f$ is defined by
\begin{equation}\label{Cesaro means}
\sigma_0(f)=f~\mbox{and}~\sigma_n(f)(\mathrm e^{\mathrm ix}):=\frac{1}{n}\sum_{k=1}^nS_k(f)(\mathrm e^{\mathrm ix}),~\mbox{if}~n=1,2,\ldots,
\end{equation}
where $S_k(f)(z):=\sum_{j=-k+1}^{k-1}c_j\mathrm e^{\mathrm ijx}$. The trigonometric polynomials $\sigma_n(f)$ also are called the Fej\'{e}r means of $f$.

The famous L. Fej\'{e}r theorem says:

\textit{\\$(i)$ for each $n=0,1,\ldots$, $\sigma_n(f)\ge 0$ on $\mathbb T$ if $f\ge0$ on $\mathbb T$;\\	
$(ii)$ for each $n=0,1,\ldots$ and $f\in C(\mathbb T)$, $\|\sigma_n(f)\|_{C(\mathbb T)}\le \|f\|_{C(\mathbb T)}$;\\
$(iii)$ if $f\in C(\mathbb T)$, then $\sigma_n(f)$ converge uniformly  to $f$ on $\mathbb T$ as $n\to\infty$.}

The classical proof is based on the representation of $\sigma_n(f)$ by convolution 
\begin{eqnarray}\label{Fejer int}
\sigma_n(f)(\mathrm e^{\mathrm ix})&=&\left(f* K_n\right)(x)\nonumber\\
&=&\frac{1}{2\pi}\int_{-\pi}^{\pi}f(\mathrm e^{\mathrm it})K_n(x-t)dt,
\end{eqnarray}
where
\[
K_n(t):=\frac{1}{n}\left(\frac{\sin\frac{nt}{2}}{\sin\frac{t}{2}}\right)^2
\]
is the approximative identity (Fej\'{e}r kernel).

If the Fourier series of $f$ is of the power series type, i. e.
\[
f(\mathrm e^{\mathrm ix})\sim\sum_{k=0}^{\infty}c_k\mathrm e^{\mathrm ikx},
\]
then, it is readily verified that 
\begin{eqnarray}\label{Fejer repres Sn}
\sigma_n(f)(z)&=&\sum_{k=0}^{n-1}\left(1-\frac{k}{n}\right)c_kz^k\nonumber\\
&=&S_n(f)(z)-\frac{z}{n}S_n'(f)(z),\quad z\in\mathbb D,
\end{eqnarray}
where $f(z)=\sum_{k=0}^{\infty}c_kz^k$ is holomorphic function in the unit disc $\mathbb D:=\{z\in\mathbb C : |z|<1\}$ and $S_n(f)(z)$ its partial Taylor's sums.
In this case we have the following assertion:

\bigskip
\noindent
\textit{$(iv)$ $n(f(z)-\sigma_n(f)(z))/z$ converge to $f'(z)$ uniformly on compact sets of $\mathbb D$.}

\bigskip
The assertion $(iv)$ is a Voronovskaja type theorem for $\sigma_n(f)$ in holomorphic case. See \cite{Gal} and \cite{Savchuk} for further generalization.

In this paper we are interested in extending of $(i)$--$(iv)$ on Fej\'{e}r type means of Fourier series expansion based on Takenaka--Malmquist orthonormal system.

For more precisely formulation of our goals, we need some notations and some well known facts.

Let $\mathcal H$ denot the set of all functions holomorphic in $\mathbb D$ and let $dm$ be the normalized Lebesgue measure on $\mathbb T$. For $1\le p\le\infty$, the Hardy space $H^p$ consist of those $f\in\mathcal H$ for which
\[
+\infty>\|f\|_p:
=\begin{cases}
\displaystyle\sup_{\rho\in[0,1)}\left(\int_\mathbb T|f(\rho t)|^pdm(t)\right)^{1/p},\hfill& 1\le p<\infty,\cr
\sup_{z\in\mathbb D}|f(z)|,\hfill& p=\infty.
\end{cases}
\]

It is well known that each function $f$ from the space $H^p$ has non-tangential limits almost everywhere on the circle $\mathbb T$ and, moreover, these limit values form a measurable function belonging to $L^p(\mathbb T)$. This function we will also denote by $f$ and  note that 
\begin{eqnarray*}
\|f\|_{L^p(\mathbb T)}&:=&\left(\int_\mathbb T|f|^pdm\right)^{1/p}\\
&=&\|f\|_p,\quad 1\le p<\infty,
\end{eqnarray*}
and
\[
\|f\|_{L^\infty(\mathbb T)}=\mathrm{ess}\sup_{t\in\mathbb T}|f(t)|.
\]

For a given system of points ${\bf a}:=\{a_k\}_{k=0}^\infty,$ $a_k\in\mathbb D$ (the points $a_k$ are enumerated taking into account their multiplicity), the Takenaka--Malmquist  system $\varphi=\{\varphi_k\}_{k=0}^\infty$ is defined by
\begin{equation}\label{phi def}
\varphi_k(z):=\frac{\sqrt{1-|a_k|^2}}{1-z\overline a_k}B_k(z),~k\in\mathbb Z_+,
\end{equation}
where
\[
B_k(z):=\begin{cases}
1,\hfill&\mbox{if}~k=0,\cr
\displaystyle \prod_{j=0}^{k-1}\frac{z-a_j}{1-z\overline a_j},\hfill&\mbox{if}~k\in\mathbb N,
\end{cases}
\] 
is Blaschke product for the disc $\mathbb D$ of degree $k$. 

It is well known that $\varphi$ is an orthonormal and complete system in $H^2$ if and only if $B_n(z)\to 0$ as $n\to\infty$ uniformly on compact sets in $\mathbb D$. The last condition is equivalent to (see \cite{Burckel}, p. 195)
\begin{equation}\label{completeness}
\sum_{k=0}^{\infty}(1-|a_k|)=+\infty.
\end{equation}
So, if (\ref{completeness}) is fulfilled, then for any function $f\in H^2$,
\[
\lim_{n\to\infty}\left\|f-\sum_{k=0}^{n}\langle f,\varphi_k\rangle\varphi_k\right\|_2=0.
\]
Consequently, we have that for any $f\in H^2$ the Fourier series
\begin{equation}\label{Fourier series}
\sum_{k=0}^{\infty}\langle f,\varphi_k\rangle\varphi_k,
\end{equation}
where
\begin{equation}\label{coeff}
\langle f,\varphi_k\rangle:=\int_\mathbb T f\overline{\varphi_k}dm,\quad k=0,1,\ldots,
\end{equation}
convergence to $f$ uniformly in $\mathbb D$ as well as in $H^2$--metric on $\mathbb T$, provided (\ref{completeness}).

Actually, we can extend the concept of Fourier series (\ref{Fourier series}) on the space of Cauchy transforms
\[
\mathcal K:=\left\{f(z)=\int_\mathbb T\frac{d\mu(t)}{1-\overline tz}: z\in\mathbb D, \mu\in M\right\},
\] 
where $M$ is the space of finite complex Borel measures on $\mathbb T$. For this one, we set
\[
\langle f,\varphi_k\rangle:=\int_\mathbb T \overline{\varphi_k}d\mu,\quad k=0,1,\ldots.
\]

In fact, we have that for $f\in\mathcal K$
\[
f(z)=\sum_{k=0}^{\infty}\langle f,\varphi_k\rangle\varphi_k(z),\quad z\in\mathbb D,
\]
provided $\varphi$ is complete Takenaka---Malmquist system.
This follows immediately from definition of $\mathcal K$ and the identity \cite{Djrbashian. 1956} (an analog of the Christoffel--Darboux formula)
\begin{equation}\label{Djarbashian}
\sum_{k=0}^{n-1}\varphi_k(z)\overline{\varphi_k(t)}=\frac{1-B_n(z)\overline{B_n(t)}}{1-z\overline t},\quad n\in\mathbb N, ~z, t\in\mathbb C.
\end{equation} 

When a function $f$ is only defined on $\mathbb T$ and $f\in L^1(\mathbb T)$, then we will define $\langle f,\varphi_k\rangle$ for every $k\in\mathbb Z$ in same manner as (\ref{coeff}), where $\varphi_{-k}(t)=\overline{t\varphi_{k-1}(t)}$ for $k\in\mathbb N$. In such case we associate a function $f$ with the Fourier series
\begin{equation}\label{Fourier series extended}
f\sim\sum_{k\in\mathbb Z}\langle f,\varphi_k\rangle\varphi_k,
\end{equation}
based on extended Takenaka--Malmquist system $\varphi=\{\varphi_k\}_{k\in\mathbb Z}$.

Through the paper, we let by $TMS$ denote the set of all Takenaka--Malmquist systems (including extended). 

The equations (\ref{Cesaro means}), (\ref{Fejer int}) and (\ref{Fejer repres Sn}) hint at three possible ways to generalize of Fej\'{e}r means based on $TMS$.

The first one is the Ces\`{a}ro $(C,1)$ means of the series (\ref{Fourier series extended}), that is
\begin{eqnarray}\label{Fejer means}
\sigma_n(f)&=&\frac{1}{n}\sum_{k=1}^{n}S_k(f)\\
&=&
\sum_{k=-n+1}^{n-1}\left(1-\frac{|k|}{n}\right)\langle{f, \varphi_k}\rangle\varphi_k,\quad n=1,2,\ldots,
\nonumber
\end{eqnarray}
where $S_k(f):=\sum_{j=-k+1}^{k-1}\langle{f, \varphi_j}\rangle\varphi_j$, $k\in\mathbb N$,  are the partial sums of the series in (\ref{Fourier series extended}).

Unfortunately, in such case, the assertions $(i), (ii)$ and $(iv)$ are not true. 

For example, let $e_0(z):=1$ and $0<a_k<1$, $k=0,1,\ldots$.
Then it follows from (\ref{Djarbashian}), that
\[
e_0(t)\sim\sum_{k=0}^{\infty}\overline{\varphi_k(0)}\varphi_k(t),\quad t\in\mathbb T,
\]
provided (\ref{completeness}) is satisfied.

Consequently,
\[
S_k(e_0)(z)=1-\overline{B_k(0)}B_k(z),
\]
and therefore, by (\ref{Fejer means}), we get
\[
\sigma_n(e_0)(z)=1-\frac{1}{n}\sum_{k=1}^{n}\overline{B_k(0)}B_k(z).
\]
It follows that
\begin{eqnarray*}
\|\sigma_n(e_0)\|_{C(\mathbb T)}&=&\left\|1-\frac{1}{n}\sum_{k=1}^{n}\overline{B_k(0)}B_k\right\|_{C(\mathbb T)}\\
&=&1+\frac{1}{n}\sum_{k=1}^{n}B_k(0)B_k(-1)\\
&=&1+\frac{1}{n}\sum_{k=1}^{n}\left(\prod_{j=0}^{k-1}a_j\right)\\
&>&\|e_0\|_{C(\mathbb T)}.
\end{eqnarray*}
So, $(i)$ and $(ii)$ are not true in general.

Analogously, if we consider $e_0$ as a holomorphic function in $\mathbb D$, we obtain that for all $z\in (-1, 0]$ and for  each natural $n$
\begin{eqnarray*}
\left|n(e_0(z)-\sigma_n(e_0)(z))-ze_0'(z)\right|&=&\left|\sum_{k=1}^{n}B_k(0)B_k(z)\right|\\
&=&\sum_{k=1}^{n}\left(\prod_{j=0}^{k-1}a_j\frac{|z|+a_j}{1+|z|a_j}\right)\\
&\ge&\sum_{k=1}^{n}\frac{1}{2^k}\left(\prod_{j=0}^{k-1}a_j\right)^2\\
&>&0.
\end{eqnarray*}
This gives that $(iv)$ is not true for all holomorphic functions.

The assertion $(iii)$ can be rescued at least for holomorphic functions. More precisely, it was proved in \cite{Van Guht} that $\sigma_n(f)$ converge uniformly to the function $f$ in $\overline{\mathbb D}$ if
the boundary function $f(\mathrm e^{\mathrm ix})$ is continuous on $\mathbb T$, provided $a_k$ are all in a compact subset of $\mathbb D$ and satisfy certain (mild) condition on the distribution. 

In this paper we consider the other two cases of generalization of Fej\'{e}r means. Namely, we consider the operators $\sigma_{n, \varphi}$ and  $\sigma^+_{n,\varphi}$ defined on $L^1(\mathbb T)$ and $\mathcal K$ respectively as follows:
\[
\sigma_{n,\varphi}(f)(\mathrm e^{\mathrm ix}):=
\begin{cases}
f(\mathrm e^{\mathrm ix}),\hfill&\mbox{if}~n=0,\cr
\displaystyle\frac{1}{4\pi\gamma_n(x)}\int_{-\pi}^{\pi}f(\mathrm e^{\mathrm iy})\frac{\displaystyle\sin^2\left(\int_x^y\gamma_n(t)dt\right)}{\displaystyle\sin^2\frac{y-x}{2}}dy,\hfill&\mbox{if}~n\in\mathbb N
\end{cases},
\quad x\in\mathbb R,
\]
where 
\[
\gamma_n(t):=\frac{1}{2}\sum_{k=0}^{n-1}\frac{1-|a_k|^2}{1-2|a_k|\cos(t-\arg a_k)+|a_k|^2},
\]
and
\[
\sigma^+_{n, \varphi}(f)(z):=
\begin{cases}
f(z),\hfil&\mbox{if}~n=0,\cr
\displaystyle S_n(f)(z)-\frac{B_n(z)}{B_n'(z)}S'_n(f)(z),\hfill&\mbox{if}~ n\in\mathbb N
\end{cases}
,\quad z\in\mathbb T,
\]
where
\[
S_n(f)=\sum_{k=0}^{n-1}\langle f,\varphi_k\rangle\varphi_k.
\]

The operator $\sigma_{n, \varphi}$ was introduced by V. Rusak \cite{Rusak1}. We mention \cite{Rusak2, Rovba1, Rovba2, Rovba3, Pekarskii} as general references for the approximation properties of $\sigma_{n,\varphi}$ and their analogous.

As far as we know, the operator $\sigma^+_{n,\varphi}$ (based on TMS) is considered here for the first time. 

In case $a_k=0$, $k\in\mathbb Z_+$, the restrictions to $H^1$ of the operators $\sigma_{n, \varphi}(f)$ and $\sigma^+_{n, \varphi}(f)$ coincide with  the usual Fej\'{e}r means $\sigma_n(f)$ of trigonometric Fourier series of $f\in H^1$ on the circle $\mathbb T$. 
Actually, as we will show later, $\sigma_{n, \varphi}(f)=\sigma^+_{n, \varphi}(f)$ for all $n\in\mathbb Z_+$ and $\varphi\in TMS$, provided $f\in H^1$. 

It should be remarked also that for $f\in \mathcal K$ the functions $\sigma^+_{n, \varphi}(f)$, $n\in\mathbb Z_+$,  are holomorphic in $\mathbb D$, provided that $B_n/B'_n$, $n\in\mathbb Z_+,$ are holomorphic in $\mathbb D$, and are meromorphic in $\mathbb D$ otherwise. But in any case, we can consider the functions
\[
\delta_{n, \varphi}(f):=
\begin{cases}
0,\hfill&\mbox{if}~n=0,\cr
\displaystyle
\frac{B'_n}{B_n}\left(f-\sigma^+_{n, \varphi}(f)\right),\hfill&\mbox{if}~n\in\mathbb N,
\end{cases}
\]
as the error functions of weighted approximation in $\mathbb D$ to $f\in H^1$ by
$\sigma^+_{n,\varphi}(f)$. In view of (\ref{Fejer int}), it is natural to conjecture that $\delta_{n,\varphi}(f)(z)\to f'(z)$ uniformly on compact sets in $\mathbb D$ as $n\to\infty$. We will show that this is really true if (\ref{completeness}) is satisfied.

Given these remarks, it is reasonable to studies the operators  $\sigma_{n, \varphi}(f)$ and $\sigma^+_{n, \varphi}(f)$ separately  on the unit circle $\mathbb T$ and  in the unit disc $\mathbb D$ respectively.

Our goals are to solve the following two main problems.

\begin{prob}\label{prob1}
Let $X(\mathbb T)$ be one of the space $C(\mathbb T)$ and $L(\mathbb T)$, $1\le p<\infty$. Find necessary and sufficient condition on $\varphi\in TMS$ in order that 	$\sigma_{n, \varphi}(f)$ converges in metric $X(\mathbb T)$ to $f$ for every function $f\in X(\mathbb T)$.
\end{prob}

\begin{prob}\label{prob2}
Let $K$ be a class of holomorphic functions in $\mathbb D$, $\varphi\in TMS$, $n\in\mathbb N$ and $z\in\mathbb D$. Find the quantity
\[
\sup_{f\in K}|\delta_{n, \varphi}(f)(z)-f'(z)|.
\]
\end{prob}

\section{Main results}

In this section we only state our main theorems and some corollaries, and the proofs of theorems will be given in Section \ref{proofs}.

Our main result in part of uniform convergence of $\sigma_{n,\varphi}$ for $C(\mathbb T)$ is connected with the well--known Frostman theorem \cite{Frostman}(see also \cite{Mashreghi}, p. 117).

\noindent
{\bf Frostman's theorem}. {\it Let $B$ be a infinite Blaschke product with zero-sequence $\{a_k\}_{k=0}^\infty$, and $t\in\mathbb T$. Then $B$ has an angular derivative in the sence of Carath\'{e}odory at $t$ (for definition see \cite{Mashreghi}, p. 62) if and only if 
\[
\sum_{k=0}^{\infty}\frac{1-|a_k|^2}{|1-t\overline a_k|^2}<\infty.
\]
}

It is shown in \cite{Shapiro} (p.185) that there is an infinite Blaschke product that has an angular derivative at no point of $\mathbb T$ or, equivalently, there is a sequance $\{a_k\}_{k=0}^\infty$ such that 
\[
\sum_{k=0}^{\infty}(1-|a_k|)<\infty,
\] 
yet 
\begin{equation}\label{Frostman con}
\sum_{k=0}^{\infty}\frac{1-|a_k|^2}{|1-t\overline a_k|^2}=\infty\quad\mbox{for all}~t\in\mathbb T.
\end{equation}

\begin{thm}\label{main thm} Suppose that $\varphi\in TMS$ and $X(\mathbb T)$ is one of the space $C(\mathbb T)$ and $L^p(\mathbb T)$, $1\le p<\infty$. In order that 	$\|f-\sigma_{n, \varphi}(f)\|_{X(\mathbb T)}\to 0$  as $n\to\infty$ for every $f\in X(\mathbb T)$ it is necessary and sufficient that
\begin{equation}\label{convergence condt}
\lim_{n\to\infty}\left\|\frac{1}{B'_n}\right\|_{X(\mathbb T)}=0.
\end{equation}
\end{thm}

\begin{col}
Let $\varphi\in TMS$. Then for every $f\in C(\mathbb T)$, $\sigma_{n,\varphi}(f)$ converges to $f$ uniformly on $\mathbb T$ if and only if (\ref{Frostman con}) is fulfilled. 
\end{col}

Indeed, in case $X(\mathbb T)=C(\mathbb T)$ the norm in the left hand side of (\ref{convergence condt}) becomes 
\begin{equation}\label{eq for C(T)}
\left\|\frac{1}{B'_n}\right\|_{C(\mathbb T)}=\frac{1}{\displaystyle\min_{t\in\mathbb T}\sum_{k=0}^{n-1}\frac{1-|a_k|^2}{|1-\overline ta_k|^2}}.
\end{equation}

\begin{col}
	Let $\varphi\in TMS$ be such that  (\ref{Frostman con}) is fulfilled. Then for every $f\in L^p(\mathbb T)$, $1\le p<\infty$, we have $\|f-\sigma_{n,\varphi}(f)\|_p\to 0$ as $n\to\infty$.
\end{col}

Indeed, similarly to (\ref{eq for C(T)}), we get
\begin{equation}\label{up Frostman}
\left\|\frac{1-\overline{B'_n(0)}B_n}{B'_n}\right\|_{X(\mathbb T)}\le\frac{2}{\displaystyle\min_{t\in\mathbb T}\sum_{k=0}^{n-1}\frac{1-|a_k|^2}{|1-\overline ta_k|^2}}
\end{equation}
and result follows.

\begin{col} Suppose that $\varphi\in TMS$ and $X(\mathbb T)$ is one of the space $C(\mathbb T)$ and $L^p(\mathbb T)$, $1\le p<\infty$. If $\varphi$ is complete, then for every $f\in X(\mathbb T)$, we have  $\|f-\sigma_{n, \varphi}(f)\|_{X(\mathbb T)}\to 0$  as $n\to\infty$.
\end{col}

This assertion readily follows from the estimates (\ref{eq for C(T)}), (\ref{Blashe<Frostman}) and (\ref{up Frostman}).

Let us consider the class $\mathcal K^+$ consisting of those holomorphic functions in $\mathbb D$, which can be represented by Cauchy type integral
\begin{equation}\label{Cauchy integral}
f(z)=K(\mu)(z):=\int_\mathbb T\frac{\mu(t)}{1-\overline tz}dm(t),\quad z\in\mathbb D,
\end{equation} 
with $\mu\in L^\infty(\mathbb T)$,  $\|\mu\|_\infty:=\mathrm{ess~sup}_{t\in\mathbb T}|\mu(t)|\le 1.$

\bigskip
The following assertion gives solution of Problem \ref{prob2}.

\begin{thm}\label{Vor thm}
	Let $\varphi\in TMS$. Then for every $z\in\mathbb D$ and $n\in\mathbb Z_+$ we have
	\begin{equation}\label{eq1}
	\max_{f\in\mathcal K^+}\left|\delta_{n, \varphi}(f)(z)-f'(z)\right|=\frac{|B_n(z)|}{1-|z|^2}.
	\end{equation}
	For given $z\in\mathbb D$ such that $B_n(z)\not=0$, and for given $n\in\mathbb Z_+$ maximum is attained only for the functions
	\begin{equation}\label{extremal 1}
	f_*(t)=\mathrm e^{\mathrm i\theta}B_n(t)\frac{t-z}{1-t\overline z},\quad\theta\in\mathbb R.
	\end{equation}
\end{thm}

As immediate consequence of Theorem \ref{Vor thm}, we have the following assertion.

\begin{col}\label{col1} Let $f\in\mathcal K^+$,  $\varphi\in TMS$ and let ${\bf a}=\{a_j\}_{j=0}^\infty$ be as in (\ref{phi def}). Then for every $n\in\mathbb N$ we have
	\begin{equation}\label{d=f}
	\delta_{n, \varphi}(f)(a_j)=f'(a_j),\quad j=0,1,\ldots, n-1.
	\end{equation}
	Moreover, if (\ref{completeness}) is satisfied, then uniformly on closed sets of $\mathbb D$
	\[
	\lim_{n\to\infty}\delta_{n, \varphi}(f)(z)=f'(z),
	\]
	and, therefore, for given $z\in\mathbb D$ the equality  
	\[
	\lim_{n\to\infty}\delta_{n, \varphi}(f)(z)=0
	\]
	hold true if and only if $f'(z)=0$.
\end{col}

Now, let us consider the behavior of $|f(z)-\sigma^+_{n,\varphi}(f)(z)|$ for $z\in\overline{\mathbb D}$.

We let by $\mathcal S$ denote the Schur class consisting of holomorphic functions $f$ which satisfy $\sup_{z\in\mathbb D}|f(z)|\le 1$ and let $A(\mathbb D)$ denote the disk algebra of holomorphic functions in $\mathbb D$ that extend continuously to $\overline{\mathbb D}$ equipped with the norm $\|f\|_{A(\mathbb D)}=\max_{z\in\overline{\mathbb D}}|f(z)|$.

\begin{thm}\label{Th3} Suppose $\varphi\in TMS$, $z\in\mathbb D$ and $n\in\mathbb N$. Then we have
\begin{eqnarray}\label{lower est}
\left|\frac{B_n(z)}{B'_n(z)}\right|\frac{1-|B_n(z)|^2}{1-|z|^2}&\le&\sup_{f\in\mathcal S}|f(z)-\sigma_{n, \varphi}^+(f)(z)|\\
&\le&
\left|\frac{B_n(z)}{B'_n(z)}\right|\frac{1-|B_n(z)|^2}{1-|z|^2}+|B_n(z)|^2,\nonumber
\end{eqnarray}
provided $z\in\mathbb D$ and $|B'_n(z)|>0$.

Moreover, we have that for all $n\in\mathbb N$
\begin{eqnarray*}
	\sup_{f\in\mathcal S}\|f-\sigma^+_{n,\varphi}(f)\|_{\infty}&=&	\sup_{f\in\mathcal S\cap A(\mathbb D)}\|f-\sigma^+_{n,\varphi}(f)\|_{A(\mathbb D)}\\
	&=&2.
\end{eqnarray*}
\end{thm}

In view of this result, it is interesting to obtain the lower estimate of $|f-\sigma^+_{n,\varphi}(f)|$ for an individual function $f$. For this goal we recall that the classical Fej\'er means $\sigma_n(f)$ considered as a method of approximation on space $C(\mathbb T)$ are saturated and its saturation order is $O(n^{-1})$ (see for example \cite{Stepanets}, p. 79). In particular, this means that {\it the relation $\|f-\sigma_n(f)\|_{C(\mathbb T)}=o(n^{-1})$ as $n\to\infty$ holds true only if $f=\mathrm{const}$}. The next corollary gives an analog of this assertion for $\sigma_{n,\varphi}$. 

\begin{col} Suppose $f\in H^\infty$, $\varphi\in TMS$ and let ${\bf a}=\{a_j\}_{j=0}^\infty$ be as in (\ref{phi def}). Then for every $n\in\mathbb N$ we have
\[
\|f-\sigma^+_{n,\varphi}(f)\|_{L^\infty(\mathbb T)}\ge\frac{1}{n}\max_{0\le j\le n-1}\left((1-|a_j|^2)\left|f'(a_j)\right|\right).
\]
	
Moreover, if (\ref{completeness}) is satisfied and if	
\[
\|f-\sigma_{n,\varphi}(f)\|_{L^\infty(\mathbb T)}=o\left(\frac{1}{n}\right),\quad n\to\infty,
\]
then $f=\mathrm{const}$.
\end{col}

For proving we fix $n\in\mathbb N$ and consider the function 
\[
F:=\frac{f-S_n(f)}{B_n}.
\] 
Since $f(a_j)-S_n(f)(a_j)=0$ for $j=0,\ldots,n-1$, we have (see \cite{Garnett}, p. 53) that $F$ belongs to $H^\infty$ space. On the other side,  $\delta_{n,\varphi}(f)=B'_nF+S'_n(f)$. Therefore $\delta_{n,\varphi}(f)\in H^\infty$ and, consequently, we can apply the following well-known inequality (see \cite{Garnett}, p. 85) 
\[
(1-|z|^2)|\delta_{n,\varphi}(f)(z)|\le\|\delta_{n,\varphi}(f)\|_1,\quad z\in\mathbb D.
\]
Hence, from Corollary \ref{col1} we get
\begin{eqnarray*}
	(1-|a_j|^2)|f'(a_j)|&=&(1-|a_j|^2)|\delta_{n, \varphi}(f)(a_j)|\\
	&\le&\|\delta_{n, \varphi}(f)\|_1\\
	&\le&\left\|\frac{B'_n}{B_n}(f-\sigma_{n,\varphi}(f))\right\|_1\\
	&\le&\|B'_n\|_1\|f-\sigma_{n,\varphi}(f)\|_{L^\infty(\mathbb T)}\\
	&=&n\|f-\sigma_{n,\varphi}(f)\|_{L^\infty(\mathbb T)}.
\end{eqnarray*}
Here we used the following identity
\begin{eqnarray*}
	\|B'_n\|_1&=&\sum_{j=0}^{n-1}\int_\mathbb T\frac{1-|a_j|^2}{|1-t\overline a_j|^2}dm(t)\\
	&=&n.
\end{eqnarray*}

The second part of assertion follows from the inequality (as above)
\[
(1-|z|^2)|\delta_{n, \varphi}(f)(z)|\le n\|f-\sigma_{n,\varphi}(f)\|_1,\quad\forall z\in\mathbb D
\] 
by Corollary \ref{col1}.

\section{Proofs}\label{proofs}

Our approach in proving the main results is based on the following lemmas concerning the some properties of operators $\sigma_{n,\varphi}$ and $\sigma^+_{n,\varphi}$.

\subsection{Auxiliary lemmas}

All lemmas in this subsection are new and of independent interest.   

We begin with the integral representation for $\delta_{n, \varphi}$ from which follows that $\delta_{n, \varphi}(f)$ is holomorphic in $\mathbb D$.

\begin{lem}\label{lem1}
	Suppose that $\varphi\in TMS$ and $f\in\mathcal K$. Then the function $\delta_{n, \varphi}(f)$ is holomorphic in $\mathbb D$. 
	
	Moreover, if $f$ is the Cauchy transform of $\mu\in M$, then for all $n\in\mathbb Z_+$ and $z\in\mathbb D$, 
	\begin{eqnarray}
	\delta_{n, \varphi}(f)(z)&=&f'(z)-B_n(z)\int_\mathbb T\frac{\overline t-\overline z}{1-\overline tz}\frac{\overline{B_n(t)}}{|1-\overline tz|^2}d\mu(t)\label{delta f}\\
	&=&\int_\mathbb T\frac{\overline t-\overline z}{1-\overline tz}\frac{1-\overline{B_n(t)}B_n(z)}{|1-\overline tz|^2}d\mu(t)\label{delta main}.
	\end{eqnarray}
\end{lem}

\begin{proof} Applying (\ref{Djarbashian}) we get 
\begin{equation}\label{delta}
\frac{B'_n(z)}{B_n(z)}\left(f(z)-S_n(f)(z)\right)=B'_n(z)\int_\mathbb T\frac{\overline{B_n(t)}}{1-\overline tz}d\mu(t)
\end{equation}
and
\begin{equation}\label{delta2}
f'(z)-S'_n(f)(z)=B'_n(z)\int_\mathbb T\frac{\overline{B_n(t)}}{1-\overline tz}d\mu(t)+B_n(z)\int_\mathbb T\frac{\overline{B_n(t)t}}{(1-\overline tz)^2}d\mu(t).
\end{equation}
By subtracting (\ref{delta2}) from (\ref{delta}) we get
\[
\delta_{n, \varphi}(f)(z)=f'(z)-B_n(z)\int_\mathbb T\frac{\overline{B_n(t)t}}{(1-\overline tz)^2}d\mu(t).
\]
This means that the function $\delta_{n, \varphi}(f)$ is holomorphic in $\mathbb D$. 

In order to prove (\ref{delta f}) and (\ref{delta main}), it remains to note that 
\[
\frac{\overline t}{(1-\overline tz)^2}=\frac{\overline t-\overline z}{1-\overline tz}\frac{1}{|1-\overline tz|^2},\quad t\in\mathbb T.
\]
\end{proof}

\begin{lem}\label{sigma+ for e0}
	Let $\varphi\in TMS$ and $e_0(t)=1$ for $t\in\mathbb T$. Then for all $n\in\mathbb Z_+$  we have $\sigma^+_{n,\varphi}(e_0)=e_0.$
\end{lem}

\begin{proof} By the Lemma \ref{lem1} (see (\ref{delta main}) for $d\mu=e_0dm$), according to the Poisson integral formula, applying to the holomorphic function $t\mapsto(t-z)(1-B_n(t)\overline{B_n(z)})/(1-t\overline z)$, we have 
\begin{eqnarray*}
	\overline{B'_n(z)}\overline{\left(\frac{e_0(z)-\sigma^+_{n,\varphi}(e_0)(z)}{B_n(z)}\right)}
	&=&\overline{\delta_{n, \varphi}(e_0)(z)}\\&=&\int_\mathbb T\frac{t-z}{1-t\overline z}\frac{1-B_n(t)\overline{B_n(z)}}{|1-\overline tz|^2}dm(t)\\
	&=&0,~\forall z\in\mathbb D.
\end{eqnarray*}
The result follows, since $B'_n$ has precisely $n-1$ zeros in $\mathbb D$ (see \cite{Mashreghi}, p. 41).
\end{proof}

\begin{lem}\label{lim B}
	Let $B$ be a finite Blacshke product for the disc $\mathbb D$ and let $f\in H^1$. Then for every $z\in\mathbb T$ we have
	\[
	\lim_{\rho\to 1-}\left(1-|B(\rho z)|^2\right)f(\rho z)=0.
	\] 
\end{lem}

\begin{proof} 
Since
\[
\lim_{\rho\to 1-}\frac{1-|B(\rho z)|^2}{1-|\rho z|^2}=|B'(z)|>0
\]
for every $z\in\mathbb T$ (see \cite{Mashreghi}, p. 45), it is enough to prove the lemma in the case $B(z)=z$.

By Cauchy integral formula, applying to the function $t\mapsto f(t)(1-\rho\overline zt)$, we have
\[
\left(1-\rho^2\right)f(\rho z)=\int_\mathbb Tf(t)\frac{1-\rho\overline zt}{1-\rho z\overline t}dm(t).
\]
Since
\[
\lim_{\rho\to 1-}\frac{1-\rho\overline zt}{1-\rho z\overline t}=
\begin{cases}
-\overline zt,\hfill&\mbox{if}~t\not=z,\\
1,\hfill&\mbox{if}~t=z,
\end{cases}
\]
by Lebesgue's dominated convergence theorem we get
\begin{eqnarray*}
	\lim_{\rho\to 1-}\left(1-\rho^2\right)f(\rho z)&=&\lim_{\rho\to 1-}\int_\mathbb Tf(t)\frac{1-\rho\overline zt}{1-\rho z\overline t}dm(t)\\
	&=&-\overline z\int_\mathbb Tf(t)tdm(t)\\
	&=&0.
\end{eqnarray*}
\end{proof}

For $\varphi\in TMS$ we define the kernels $\mathcal F_{n,\varphi}(t, z)$, $n\in\mathbb Z_+$, as the functions on $\mathbb T^2$ by
\[
\mathcal F_{n,\varphi}(t, z):=
\begin{cases}
1,\hfil&\mbox{if}~n=0,\cr
\displaystyle2\mathrm{Re}\sum_{k=0}^{n-1}\left(\varphi_k(z)-\frac{B_n(z)}{B'_n(z)}\varphi'_k(z)\right)\overline{\varphi_k(t)}-1,\hfill&\mbox{if}~ n\in\mathbb N.
\end{cases}
\]
Since $B'_n(z)\not=0$ for $z\in\mathbb T$ and $n\in\mathbb N$ (see \cite{Mashreghi}, p. 40), the kernels $\mathcal F_{n,\varphi}(t, z)$ are well defined.

We call $\mathcal F_{n,\varphi}$ \textit{the Fej\'{e}r--type kernel of order $n$} based on the system $\varphi\in TMS$. The reason for this is that in case $\varphi_k(t)=t^k$, $\mathcal F_{n,\varphi}$ coincide with the Fej\'{e}r kernel
\[
F_n(\overline tz):=
\begin{cases}
1,\hfill&\mbox{if}~n=0,\cr
\displaystyle
2\mathrm{Re}\sum_{k=0}^{n-1}\left(1-\frac{k}{n}\right)z^k\overline t^k-1,\hfill&\mbox{if}~n\in\mathbb N.
\end{cases}
\]
In addition, as we will show in the following Lemma \ref{Fejer kern}, $\mathcal F_{n,\varphi}$ is positive kernel with
\begin{equation}\label{mean value}
\int_\mathbb T\mathcal F_{n,\varphi}(t, z)dm(t)=1,\quad z\in\mathbb T.
\end{equation}

\begin{lem}\label{Fejer kern}
	Suppose that $\varphi\in TMS$, $f\in H^1$ and $n\in\mathbb N$. Then for every $z\in\mathbb T$ we have
	\begin{eqnarray}\label{Fejer+}
	\sigma^+_{n,\varphi}(f)(z)&=&c_n(z)\int_\mathbb Tf(t)\left|\frac{B_n(t)-B_n(z)}{t-z}\right|^2dm(t)\\
	&=&\sigma_{n,\varphi}(f)(z),\label{Fejer kern 2}
	\end{eqnarray}
	where
	\begin{equation}\label{cn def}
	c_n(z):=\frac{1}{\sum_{k=0}^{n-1}|\varphi_k(z)|^2}=\frac{1}{|B'_n(z)|}.
	\end{equation}
	Moreover, if $(t,z)\in\mathbb T^2$ and $x=\arg z,$ $y=\arg t$ we have
	\begin{eqnarray}\label{F=||^2}
	\mathcal F_{n,\varphi}(t, z)&=&c_n(z)\left|\frac{B_n(t)-B_n(z)}{t-z}\right|^2\\
	&=&\frac{1}{2\gamma_n(x)}\cdot\frac{\displaystyle\sin^2\left(\int_x^y\gamma_n(t)dt\right)}{\displaystyle\sin^2\frac{y-x}{2}}.\nonumber
	\end{eqnarray}
\end{lem}

\begin{proof} Fix $z\in\mathbb D$ and consider the holomorphic function $F_{f,z}(t):=f(t)(1-\overline{B_n(z)}B_n(t))$. 
Observe that
\begin{equation}\label{sigms(F)=sigma(f)}
\sigma^+_{n, \varphi}(F_{f,z})=\sigma^+_{n, \varphi}(f).
\end{equation} 
This follows easily from the equalities
\begin{eqnarray*}
	S_n(fB_n)(z)&=&\int_\mathbb Tf(t)B_n(t)\frac{1-B_n(z)\overline{B_n(t)}}{1-\overline tz}dm(t)\\
	&=&\int_\mathbb Tf(t)(B_n(t)-B_n(z))\frac{1}{1-\overline tz}dm(t)\\
	&=&0,\quad\forall~z\in\mathbb D.
\end{eqnarray*}

Thus applying (\ref{delta main}) to $F_{f,z}$  and taken into account (\ref{sigms(F)=sigma(f)}), we get
\begin{eqnarray}\label{represent1}
\frac{B'_n(z)}{B_n(z)}\left((1\right.&-&\left.|B_n(z)|^2)f(z)-\sigma^+_{n, \varphi}(f)(z)\right)\nonumber\\
&=&\int_\mathbb Tf(t)\frac{\overline t-\overline z}{1-\overline tz}\left|\frac{1-\overline{B_n(t)}B_n(z)}{1-\overline tz}\right|^2dm(t)\\
&=&\int_\mathbb Tf(t)\frac{\overline t-\overline z}{1-\overline tz}\left|\sum_{k=0}^{n-1}\varphi_k(z)\overline{\varphi_k(t)}\right|^2dm(t).\label{represent2}
\end{eqnarray}
Since for $(t, z)\in\mathbb T^2$,
\[
\lim_{\rho\to 1-}\frac{\overline t-\rho\overline z}{1-\overline t\rho z}=
\begin{cases}
-\overline z,\hfill&\mbox{if}~t\not=z,\\
\overline z,\hfill&\mbox{if}~t=z,
\end{cases}
\]
we can go to the limit as $|z|\to 1$ in both side of (\ref{represent1}) and (\ref{represent2}). As a result, by Lemma \ref{lim B} and by Lebesgue's dominated convergence theorem we get
\begin{eqnarray*}
	\frac{zB'_n(z)}{B_n(z)}\sigma^+_{n,\varphi}(f)(z)&=&\int_\mathbb Tf(t)\left|\frac{B_n(t)-B_n(z)}{t-z}\right|^2dm(t)\\
	&=&\int_\mathbb T f(t)\left|\sum_{k=0}^{n-1}\varphi_k(z)\overline{\varphi_k(t)}\right|^2dm(t),\quad\forall z\in\mathbb T.
\end{eqnarray*}
By Lemma \ref{sigma+ for e0} and the Parseval identity, the last equalities implies
\begin{eqnarray*}
	\frac{zB'_n(z)}{B_n(z)}&=&\frac{zB'_n(z)}{B_n(z)}\sigma^+_{n,\varphi}(e_0(z))\\
	&=&\int_\mathbb T\left|\frac{1-B_n(z)\overline{B_n(t)}}{1-z\overline t}\right|^2dm(t)\\
	&=&\int_\mathbb T \left|\sum_{k=0}^{n-1}\varphi_k(z)\overline{\varphi_k(t)}\right|^2dm(t)\\
	&=&\sum_{k=0}^{n-1}|\varphi_k(z)|^2.
\end{eqnarray*}
Since $zB'_n(z)/B_n(z)=|B'_n(z)|$ (see \cite{Mashreghi}, p. 40), (\ref{Fejer+}) follows.

In order to prove (\ref{Fejer kern 2}) and (\ref{F=||^2}), we fix $z\in\mathbb T$ and  consider the holomorphic function
\[
g(t)=\sum_{k=0}^{n-1}\overline{\left(\varphi_k(z)-\frac{B_n(z)}{B'_n(z)}\varphi'_k(z)\right)}\varphi_k(t),\quad t\in\mathbb D.
\]
Differentiating (\ref{Djarbashian}) with respect to $z$ and then taking the conjugate, we obtaine the formula
\[
\sum_{k=0}^{n-1}\overline{\varphi'_k(z)}\varphi_k(t)=-\frac{\overline{B'_n(z)}B_n(t)}{1-\overline zt}+\frac{(1-\overline{B_n(z)}B_n(t))t}{(1-\overline zt)^2}.
\]
This formula and (\ref{Djarbashian}) together give us  
\begin{eqnarray*}
	g(0)&=&1-\overline{B_n(z)}B_n(0)+\overline{\frac{B_n(z)}{B'_n(z)}B'_n(z)}B_n(0)\\
	&=&1,
\end{eqnarray*}
and consequently, \begin{eqnarray*}
	\int_\mathbb Tf\cdot(g-1)dm&=&f(0)(g(0)-1)\\
	&=&0,\quad\forall~f\in H^1.
\end{eqnarray*}
Therefore,
\begin{eqnarray*}
	\sigma^+_{n,\varphi}(f)(z)&=&\int_\mathbb Tf(t)\overline{g(t)}dm(t)\\
	&=&\int_\mathbb Tf(t)\left(\overline{g(t)}+g(t)-1\right)dm(t)\\
	&=&\int_\mathbb Tf(t)\mathcal F_{n,\varphi}(t,z)dm(t)
\end{eqnarray*}
and (\ref{Fejer kern 2}) is proved.

Now let us consider the difference 
\[
\Delta_{n,z}(t):=\mathcal F_{n,\varphi}(t,z)-c_n(z)\left|\sum_{k=0}^{n-1}\varphi_k(z)\overline{\varphi_k(t)}\right|^2
\]
as the function defined on $\mathbb T$. From (\ref{Fejer+}) and (\ref{Fejer kern 2}) we see that $\Delta_{n,z}$	
satisfy 
\[
\int_\mathbb Tf\Delta_{n,z}dm=0
\] 
for every $f\in H^\infty$. This means that $\Delta_{n,z}$ is the nontangential limits of some function from $H^1$, say $\Delta_{n,z}$, with $\Delta_{n,z}(0)=0$ (see \cite{Garnett}, p. 85). But $\mathrm{Im}\Delta_{n,z}(t)=0$. Thus by Schwarz integral formula, $\Delta_{n,z}(t)=0$ for every $t\in\overline{\mathbb D}$. 
\end{proof}

\begin{lem}\label{norm Fejer}
	Suppose that $X(\mathbb T)$ is one of space $L^p(\mathbb T)$ and $C(\mathbb T)$ and that $X_+(\mathbb D)$ is one of space $H^p$ and $A(\mathbb D)$. Then for any $\varphi\in TMS$ and  for $n\in\mathbb N$ we have
	\[
	\|\sigma_{n,\varphi}\|_{X(\mathbb T)\rightarrow X(\mathbb T)}=\|\sigma^+_{n,\varphi}\|_{X_+(\mathbb D)\rightarrow X_+(\mathbb D)}=1.
	\]
\end{lem}

\begin{proof} The result follows immediately from Lemma \ref{sigma+ for e0} and Lemma \ref{Fejer kern}. 
\end{proof}

\begin{lem}\label{converg for e1}
	Let $\varphi\in TMS$, $\alpha\in\mathbb D$ be fixed and 
	\[
	w_\alpha(z):=\frac{z-\alpha}{1-z\overline\alpha}.
	\]
	Then for every $z\in\mathbb D$ and $n\in\mathbb N$ we have
	\begin{equation}\label{sigma for e1}
	\sigma_{n, \varphi}^+(w_\alpha)(z)=
	\begin{cases}
	\displaystyle z-\frac{B_n(z)}{B'_n(z)}\left(1-\overline{B'_n(0)}B_n(z)\right),\hfill&\mbox{if}~\alpha=0,\\
	\displaystyle \frac{z-\alpha}{1-z\overline\alpha}-\frac{1-|\alpha|^2}{(1-z\overline\alpha)^2}\frac{B_n(z)}{B_n'(z)}\left(1-\overline{B_n(\alpha)}B_n(z)\right),\hfill&\mbox{if}~\alpha\not=0.
	\end{cases}
	\end{equation}	
\end{lem}

\begin{proof} Let us consider case $\alpha=0$. Differentiating (\ref{Djarbashian}) with respect to $\overline t$ and then letting $t=0$ we get
\begin{equation}\label{partial sum e1}
\sum_{k=0}^{n-1}\varphi_k(z)\overline{\varphi'_k(0)}=z-\left( \overline{B'_n(0)}+\overline{B_n(0)}z\right)B_n(z).
\end{equation}
Differentiating this equality with respect to $z$ gives
\begin{equation}\label{partial sum e'1}
\sum_{k=0}^{n-1}\varphi'_k(z)\overline{\varphi'_k(0)}=1-\overline{B_n(0)}B_n(z)-\left(\overline{B'_n(0)}+\overline{B_n(0)}z\right)B'_n(z).
\end{equation}
From (\ref{partial sum e1}) we also see that 
\begin{eqnarray*}
	\langle w_0,\varphi_j\rangle&=&\int_\mathbb T\left(\sum_{k=0}^{n-1}\varphi_k(t)\overline{\varphi'_k(0)}+\left( \overline{B'_n(0)}+\overline{B_n(0)}t\right)B_n(t)\right)\overline{\varphi_j(t)}dm(t)\\
	&=&\overline{\varphi'_j(0)}+\int_\mathbb T\left( \overline{B'_n(0)}+\overline{B_n(0)}t\right)\prod_{l=j}^{n-1}\frac{t-a_l}{1-t\overline a_l}\frac{\sqrt{1-|a_j|^2}}{1-\overline ta_j}dm(t)\\
	&=&\overline{\varphi'_j(0)},\quad j=0,1,\ldots, n-1.
\end{eqnarray*}
Therefore, the sums in the left-hand side of (\ref{partial sum e1}) and (\ref{partial sum e'1}) are $S_n(w_0)(z)$ and $S'_n(w_0)(z)$ respectively. 

After some calculations, we obtain
\begin{eqnarray*}
	\sigma^+_{n,\varphi}(w_0)(z)&=&\sum_{k=0}^{n-1}\overline{\varphi'_k(0)}\varphi_k(z)-\frac{B_n(z)}{B'_n(z)}\sum_{k=0}^{n-1}\overline{\varphi'_k(0)}\varphi'_k(z)\\
	&=&z-\frac{B_n(z)}{B'_n(z)}\left(1-\overline{B'_n(0)}B_n(z)\right).
\end{eqnarray*}

Now let us consider case $\alpha\not=0$. According to (\ref{Djarbashian}) and to the identity
\[
w_\alpha(z)=-\frac{1}{\overline\alpha}+\frac{1-|\alpha|^2}{\overline\alpha}\frac{1}{1-t\overline\alpha},
\]
we have, similarly to the above,
\begin{eqnarray*}
	S_n\left(w_\alpha\right)(z)&=&-\frac{1}{\overline\alpha}S_N\left(e_0\right)(z)+\frac{1-|\alpha|^2}{\overline\alpha}S_n\left(\frac{1}{1-\cdot\overline\alpha}\right)(z)\\
	&=&\frac{1}{\overline\alpha}\sum_{k=0}^{n-1}\left(\left(1-|\alpha|^2\right)\overline{\varphi_k(\alpha)}-\overline{\varphi_k(0)}\right)\varphi_k(z)\\
	&=&\frac{1-|\alpha|^2}{\overline\alpha}\frac{1-\overline{B_n(\alpha)}B_n(z)}{1-z\overline \alpha}-\frac{1}{\overline\alpha}\left(1-\overline{B_n(0)}B_n(z)\right)
\end{eqnarray*}
and
\begin{eqnarray*}
	S_n'\left(w_\alpha\right)(z)&=&\frac{1-|\alpha|^2}{\overline\alpha}\left(\frac{-\overline{B_n(\alpha)}B_n'(z)}{1-z\overline\alpha}+\overline\alpha\frac{1-\overline{B_n(\alpha)}B_n(z)}{(1-z\overline\alpha)^2}\right)\\
	&&+\frac{1}{\overline\alpha}\overline{B_n(0)}B_n'(z).
\end{eqnarray*}

The result follows from the above with a little algebra.
\end{proof}

\begin{lem}\label{lem3}
	Let $\varphi\in TMS$, $z\in\overline{\mathbb D}$ and $n\in\mathbb Z_+$. Then we have
	\begin{equation}\label{eq2}
	\max_{f\in\mathcal S}\left|\delta_{n, \varphi}(f)(z)-B_n'(z)\overline{B_n(z)}f(z)\right|=
	\begin{cases}
	\displaystyle
	\frac{1-|B_n(z)|^2}{1-|z|^2},\hfill&\mbox{if}~z\in\mathbb D,\cr
	|B'_n(z)|,\hfill&\mbox{if}~z\in\mathbb T.
	\end{cases}
	\end{equation}
	For given $z\in\overline{\mathbb D}$ and $n\in\mathbb Z_+$ maximum is attained only for the functions
	\[
	f(t)=\mathrm e^{\mathrm i\theta}
	\begin{cases}
	\displaystyle
	\frac{t-z}{1-t\overline z},\hfill&\mbox{if}~z\in\mathbb D,\\
	1,\hfill&\mbox{if}~z\in\mathbb T,
	\end{cases}
	\quad\theta\in\mathbb R.
	\]	
\end{lem}

\begin{proof}
Fix $z\in\mathbb D$ and consider the presentations (\ref{represent1}) and (\ref{represent2}). 
It follows that
\begin{eqnarray*}
	\left|\delta_{n, \varphi}(f)(z)-B_n'(z)\overline{B_n(z)}f(z)\right|&=&\left|\frac{B'_n(z)}{B_n(z)}\left((1-|B_n(z)|^2)f(z)-\sigma^+_{n, \varphi}(f)(z)\right)\right|\\
	&\le&\int_\mathbb T\left|f(t)\frac{\overline t-\overline z}{1-\overline tz}\right|\left|\frac{1-\overline{B_n(t)}B_n(z)}{1-\overline tz}\right|^2dm(t)\\
	&=&\int_\mathbb T\left|\frac{1-\overline{B_n(t)}B_n(z)}{1-\overline tz}\right|^2dm(t)
	\\
	&=&\frac{1-|B_n(z)|^2}{1-|z|^2}.
\end{eqnarray*}
Equalities hold if and only if
\[
f(t)=\mathrm e^{\mathrm i\theta}\frac{t-z}{1-t\overline z}
\]
for some $\theta\in\mathbb R$.

Now, let us fix $z\in\mathbb T$. Then we have
\[
\left|\delta_{n, \varphi}(f)(z)-B_n'(z)\overline{B_n(z)}f(z)\right|=\left|B'_n(z)\right||\sigma^+_{n,\varphi}(f)(z)|,
\]
and the result follows by Lemma \ref{norm Fejer}.
\end{proof}

\subsection{Proof of Theorem \ref{main thm}}

First of all, note that by Lemma \ref{converg for e1}
\[
\left\|w_0-\sigma^+_{n,\varphi}(w_0)\right\|_{X(\mathbb T)}=\left\|\frac{1-\overline{B_n'(0)}B_n}{B'_n}\right\|_{X(\mathbb T)},
\]
where $w_0(z)=z$.

To estimate the left side of the last equation we observe that
\[
\prod_{k=0}^{n-1}|a_k|^2\le\left|1-\overline{B'_n(0)}B_n(z)\right|\le 2
\]
for $z\in\mathbb T$.
Here we used the well known estimate $|B'_n(0)|\le1-|B_n(0)|^2$.

Therefore
\begin{equation}\label{two side est}
\prod_{k=0}^{n-1}|a_k|^2\left\|\frac{1}{B'_n}\right\|_{X(\mathbb T)}\le\left\|w_0-\sigma^+_{n,\varphi}(w_0)\right\|_{X(\mathbb T)}\le2\left\|\frac{1}{B'_n}\right\|_{X(\mathbb T)}.
\end{equation}

To prove the necessity of (\ref{convergence condt}), suppose that $\prod_{k=0}^{n}|a_k|^2\to0$ as $n\to\infty$, yet $\lim_{n\to\infty}\|1/B'_n\|_{X(\mathbb T)}>0$. In this case, $a_k$ must satisfies the condition (\ref{completeness}). But
\begin{eqnarray}\label{Blashe<Frostman}
\frac{1}{2}\sum_{k=0}^{n}(1-|a_k|)&\le&\min_{t\in\mathbb T}\sum_{k=0}^{n}\frac{1-|a|^2}{|1-\overline ta_k|^2}\\
&\le&|B'_n(t)|\nonumber
\end{eqnarray}
for $t\in\mathbb T$.
Therefore
\[
\left\|\frac{1}{B'_n}\right\|_{X(\mathbb T)}\le\frac{2}{\sum_{k=0}^{n}(1-|a_k|)}\to 0,
\]
as $n\to\infty$. This gives a contradiction. The necessity of (\ref{convergence condt}) follows.

To prove the sufficiency, we require the following Curtis's generalization of Korovkin's theorem  \cite{Curtis} (see also \cite{Dzyadyk}).

\bigskip
\noindent
{\bf Curtis's theorem} {\it Let $T_n$ be a uniformly bounded sequence of positive operators in $X(\mathbb T)$. Then $\lim_{n\to\infty}\|f-T_n(f)\|_{X(\mathbb T)}$ for each $f\in X(\mathbb T)$ provided that $\lim_{n\to\infty}\|e_k-T_n(e_k)\|_{X(\mathbb T)}=e_k$ for $k=0,1$, where $e_k(t)=t^k$. }

\bigskip
To apply this theorem, we note that an analogues of the assertions $(i)$ and $(ii)$ for the operator $\sigma_{n,\varphi}$ in $X(\mathbb T)$ follow by the definition of $\sigma_{n,\varphi}$ and Lemma \ref{norm Fejer}:

{\it
$(i)'$ for each $n=0,1,\ldots$, $\sigma_{n,\varphi}(f)\ge 0$ on $\mathbb T$ if $f\ge0$ on $\mathbb T$;

$(ii)'$ for each $n=0,1,\ldots$, $\|\sigma_{n,\varphi}(f)\|_{X(\mathbb T)}\le \|f\|_{X(\mathbb T)}$.
}

\noindent
The assertions $(i)'$  and $(ii)'$ together with Lemma \ref{sigma+ for e0} and Curtis's theorem, where we take $T_n=\sigma_{n,\varphi}$, proves the sufficiency of Theorem \ref{main thm} because $\lim_{n\to\infty}\|e_1-\sigma_{n,\varphi}(e_1)\|_{X(\mathbb T)}=0$ according to (\ref{two side est}).

\subsection{Proof of Theorem \ref{Vor thm}}
 
It follows from (\ref{delta f}) that
\begin{eqnarray*}
	\left|\delta_{n, \varphi}(f)(z)-f'(z)\right|&\le&|B_n(z)|\int_\mathbb T\left|\mu(t)\frac{\overline t-\overline z}{1-\overline tz}\overline{B_n(t)}\right|\frac{1}{|1-\overline tz|^2}dm(t)\label{delta-f}\\
	&=&|B_n(z)|\int_\mathbb T|\mu(t)|\frac{dm(t)}{|1-\overline tz|^2}\\
	&\le&\frac{|B_n(z)|}{1-|z|^2}
\end{eqnarray*}
for every functions $f\in\mathcal K^+$ and for every $z\in\mathbb D$. 

If $B_n(z)\not=0$, then equalities hold throughout this chain of relation if and only if 
\[
|\mu(t)|=1
\] 
and
\[
\arg\left(\mu(t)\frac{\overline t-\overline z}{1-\overline tz}\overline{B_n(t)}\right)=\mathrm{const}
\]
almost everywhere on $\mathbb T$.

But this conditions are equivalent to the condition
\[
\mu(t)=\mathrm e^{\mathrm i\theta}B_n(t)\frac{t-z}{1-t\overline z}\quad\mbox{a.e. on}~\mathbb T
\]
for some $\theta\in\mathbb R$. So, the only the functions
\[
f_*(t)=K(\mu)(t)=\mathrm e^{\mathrm i\theta}B_n(t)\frac{t-z}{1-t\overline z},\quad\theta\in\mathbb R,
\]
are extremal.

\subsection{Proof of Theorem \ref{Th3}}

Let us note that the assertion is trivial in case $z\in\{a_0,\ldots, a_{n-1}\}$. Indeed, in this case we have $f(a_j)-\sigma_{n, \varphi}^+(f)(a_j)=B_n(z)=0$, therefore (\ref{lower est}) becomes equality. 

Fix $z\in\mathbb D\setminus\left(\{a_0,\ldots, a_{n-1}\}\cup\{z\in\mathbb D :  B'_n(z)=0\}\right)$. Then it follows from Lemma \ref{lem3} that for arbitrary function $f\in\mathcal S$ 
\begin{eqnarray*}
	|f(z)-\sigma^+_{n,\varphi}(f)(z)|&\le&|(1-|B_n(z)|^2)f(z)-\sigma^+_{n,\varphi}(f)(z)|+|B^2_n(z)f(z)|\\
	&=&\left|\frac{B_n(z)}{B'_n(z)}\right||\delta_{n, \varphi}(f)(z)-B_n'(z)\overline{B_n(z)}f(z)|+|B^2_n(z)f(z)|\\
	&\le&\left|\frac{B_n(z)}{B'_n(z)}\right|\frac{1-|B_n(z)|^2}{1-|z|^2}+|B_n(z)|^2.
\end{eqnarray*}
On the other side, applying (\ref{represent1}) and (\ref{represent2}) to the function $f(t)=(t-z)/(1-t\overline z)$, we get
\begin{eqnarray*}
	\frac{B'_n(z)}{B_n(z)}\left(f(z)-\sigma^+_{n, \varphi}(f)(z)\right)&=&-\frac{B'_n(z)}{B_n(z)}\sigma^+_{n, \varphi}(f)(z)\\
	&=&\frac{1-|B_n(z)|^2}{1-|z|^2}
\end{eqnarray*}
and (\ref{lower est}) follows.

Now fix $z\in\mathbb T$. Then by Lemma \ref{norm Fejer} we get
\begin{eqnarray*}
	\sup_{f\in\mathcal S}|f(z)-\sigma_{n, \varphi}^+(f)(z)|&\le&\sup_{f\in\mathcal S}|f(z)|+\sup_{f\in\mathcal S}|\sigma_{n, \varphi}^+(f)(z)|\\
	&\le&2.
\end{eqnarray*}
In order to prove the lower estimate, consider the sequence $\{f_\lambda\}_{\quad 0<\lambda<1}$ of functions
\[
f_\lambda(t)=\overline z\frac{t-\lambda z}{1-t\lambda\overline z}.
\]
It is clear that $f_\lambda\in\mathcal S$ and $|f_\lambda(t)|=1$ for all $t\in\mathbb T$. According to (\ref{Djarbashian}) and to the identity
\[
\frac{t-\omega}{1-t\overline\omega}=-\frac{1}{\overline\omega}+\frac{1}{\overline\omega}(1-|\omega|^2)\frac{1}{1-t\overline\omega},\quad 0<|\omega|<1,
\]
we get
\begin{eqnarray*}
	S_n\left(f_\lambda\right)(t)&=&-\frac{1}{\lambda}S_N\left(e_0\right)(z)+\frac{1}{\lambda}(1-\lambda^2)S_n\left(\frac{1}{1-t\lambda\overline z}\right)\\
	&=&\frac{1}{\lambda}\sum_{k=0}^{n-1}\left(\left(1-\lambda^2\right)\overline{\varphi_k(\lambda z)}-\overline{\varphi_k(0)}\right)\varphi_k(t)\\
	&=&\frac{1}{\lambda}(1-\lambda^2)\frac{1-\overline{B_n(\lambda z)}B_n(t)}{1-t\lambda\overline z}-\frac{1}{\lambda}\left(1-\overline{B_n(0)}B_n(t)\right)\end{eqnarray*}
and, consequence,
\[
S_n'\left(f_\lambda\right)(t)=\frac{1}{\lambda}(1-\lambda^2)\left(\frac{-\overline{B_n(\lambda z)}B_n'(t)}{1-t\lambda\overline z}+\lambda\overline z\frac{1-\overline{B_n(\lambda z)}B_n(t)}{(1-t\lambda\overline z)^2}\right)+\frac{1}{\lambda}\overline{B_n(0)}B_n'(t).
\]

With a little algebra it follows that
\[
f_\lambda(z)-\sigma_{n, \varphi}^+\left(f_\lambda\right)(z)=\overline z\frac{1+\lambda}{1-\lambda}\frac{B_n(z)}{B_n'(z)}\left(1-\overline{B_n(\lambda z)}B_n(z)\right).
\]

Therefore
\begin{eqnarray*}
	\left|f_\lambda(z)-\sigma_{n, \varphi}^+\left(f_\lambda\right)(z)\right|&=&\frac{1+\lambda}{1-\lambda}\frac{1}{|B_n'(z)|}\left|1-\overline{B_n(\lambda z)}B_n(z)\right|\\
	&\ge&\frac{(1+\lambda)^2}{1+|B_n(\lambda z)|}\frac{1-|B_n(\lambda z)|^2}{|B_n'(z)|(1-\lambda^2)}.
\end{eqnarray*}
Since
\[
\lim_{\lambda\to 1}\frac{1-|B_n(\lambda z)|^2}{1-\lambda^2}=|B_n'(z)|,
\]
the result follows from the above relation by letting $\lambda\to 1$.

\bibliographystyle{amsplain}

\end{document}